\documentclass[11pt]{amsart}
\usepackage{amsfonts}
\usepackage{amsfonts,bbm,mathrsfs,amssymb,amsmath,amsthm}
\newtheorem{thm}{Theorem}[section]
\newtheorem{theorem}{Theorem}[section]

\newtheorem{lemma}{Lemma}[section]

\newtheorem{prob}{Problem}[section]
\newtheorem*{cconj}{Chern Conjecture}
\newtheorem*{spp}{The Second Pinching Problem}
\theoremstyle{definition}

\theoremstyle{remark}
\newtheorem*{rmk}{Remark}

\DeclareMathOperator{\tr}{Trace}
\DeclareMathOperator{\dv}{Div}
\DeclareMathOperator{\Hess}{Hess}

\begin{document}

\title[New result on Chern conjecture]{New result on Chern conjecture for minimal hypersurfaces and its application}
\dedicatory{Dedicated to Professor Shiing-shen Chern on the occasion of his 105th birthday}
\thanks{Research supported by the National Natural Science Foundation of China, Grant Nos. 11531012, 11371315; and
the China Postdoctoral Science Foundation, Grant No. 2016M590530.}

\author{Hongwei Xu}
\author{Zhiyuan Xu}

\address{Center of Mathematical Sciences \\ Zhejiang University \\ Hangzhou 310027 \\ China}
\email{xuhw@cms.zju.edu.cn; srxwing@zju.edu.cn}
\date{}

\numberwithin{equation}{section}

\maketitle

\begin{abstract}
We verify that if $M$ is a compact minimal hypersurface in
$\mathbb{S}^{n+1}$ whose squared length of the second fundamental
form satisfying $0\leq |A|^2-n\leq\frac{n}{22}$, then $|A|^2\equiv
n$ and $M$ is a Clifford torus. Moreover, we prove that if $M$ is a
complete self-shrinker with polynomial volume growth in $\mathbb{R}^{n+1}$ whose
equation is given by (\ref{selfshr}), and if the squared length of
the second fundamental form of $M$ satisfies
$0\leq|A|^2-1\leq\frac{1}{21}$, then $|A|^2\equiv1$ and $M$ is a
round sphere or a cylinder. Our results improve the
rigidity theorems due to Q. Ding and Y. L. Xin \cite{DX1,DX2}.
\end{abstract}

\section{Introduction}\label{sec1}

The famous Chern Conjecture for minimal hypersurfaces in a sphere
was proposed by S. S. Chern \cite{Chern,CDK} in 1968 and 1970, and
was listed in the Problem Section by S. T. Yau \cite{Y2} in 1982, as
stated
\begin{cconj}\label{cj}
{\rm (A)(Standard version)} Let $M$ be a compact minimal
hypersurface with constant scalar curvature in the unit sphere
$\mathbb{S}^{n+1}$. Then the possible values of the scalar curvature
of $M$ form a
discrete set. \\
{\rm (B)(Refined version)} Let $M$ be a compact minimal hypersurface
with constant scalar curvature in the unit sphere
$\mathbb{S}^{n+1}$. Then $M$ is isoparametric. \\
{\rm (C)(Stronger version)} Let $M$ be a compact minimal
hypersurface in the unit sphere $\mathbb{S}^{n+1}$. Denote by $A$
the second fundamental form of $M$. Set $a_k=(k-sgn(5-k))n$, for
$k\in \{m\in\mathbb{Z}^+; 1\leq m\leq 5\,\}$. Then we have\\
(i) For any fixed $k\in \{m\in\mathbb{Z}^+; 1\leq m\leq 4\,\}$, if
$a_k\leq |A|^2\leq a_{k+1}$, then $M$ is isoparametric, and
$|A|^2\equiv a_k$ or $|A|^2\equiv a_{k+1}$.\\
(ii) If $|A|^2\geq a_{5}$, then $M$ is isoparametric, and
$|A|^2\equiv a_5$.
\end{cconj}

It is well-known that the Chern Conjecture consists of serval
pinching problems. In the late 1960s, Simons, Lawson, and Chern-do
Carmo-Kobayashi \cite{CDK,L1,S} solved the first pinching problem
for compact minimal hypersurfaces in the unit sphere
$\mathbb{S}^{n+1}$ and verified that if $|A|^2\leq n$, then
$|A|^2\equiv0$ and $M$ is the great sphere $\mathbb{S}^{n}$, or
$|A|^2\equiv n$ and $M$ is one of the Clifford torus
$\mathbb{S}^{k}(\sqrt{\frac{k}{n}})\times
\mathbb{S}^{n-k}(\sqrt{\frac{n-k}{n}})$, $\,1\le k\le n-1$. More
generally, they proved a rigidity theorem for compact minimal
submanifolds in a sphere under a sharp pinching condition. Further
developments on the first pinching problem have been made by many
other authors \cite{ChengNa,D,GXXZ,LL,Lu,Shen,X0,X1,Y1}, etc.

As a special part of the Chern Conjecture, the following problem has
been open for more than 40 years.
\begin{spp}
Let $M$ be a compact minimal hypersurface in the unit sphere $\mathbb{S}^{n+1}$.\\
(i) If $|A|^2$ is constant, and if $n\leq |A|^2\leq 2n$, then $|A|^2=n$, or $|A|^2=2n$.\\
(ii) If $n\leq |A|^2\leq 2n$, then $|A|^2\equiv n$, or $|A|^2\equiv 2n$.
\end{spp}

In 1983, Peng and Terng \cite{PT1,PT2} initiated the study of the
second pinching problem for minimal hypersurfaces in a unit sphere,
and made the following breakthrough on the Chern Conjecture.
\begin{theorem}\label{thm1}
Let $M$ be a compact minimal hypersurface in the unit sphere $\mathbb{S}^{n+1}$.\\
(i) If $|A|^2$ is constant, and if $n\leq |A|^2\leq n+\frac{1}{12n}$, then
$|A|^2=n$.\\
(ii) If $n\le 5$, and if
$n\leq |A|^2\leq n+\tau_1(n)$, where
$\tau_1(n)$ is a positive constant depending only on $n$, then
$|A|^2\equiv n$.
\end{theorem}

During the past three decades, there have been some important
progresses on the Chern Conjecture
\cite{AB,C,C2,C3,Cheng,DX1,GeTang,Mu,SWY,SY,Ver,WX,XT,XX1,XX2,YC0,YC1,YC2,Zhang1},
etc. In 1993, Chang \cite{C} proved Chern Conjecture (A) in
dimension three. Yang-Cheng \cite{YC0,YC1,YC2} improved the pinching
constant $\frac{1}{12n}$ in Theorem A(i) to $\frac{n}{3}$. Later,
Suh-Yang \cite{SY} improved this pinching constant to
$\frac{3}{7}n$.

\par In 2007, Wei and Xu \cite{WX} proved that if $M$ is a compact minimal hypersurface in
$\mathbb{S}^{n+1}$, $n=6, 7$, and if $n\leq |A|^2\leq n+ \tau_2(n)$,
where $\tau_2(n)$ is a positive constant depending only on $n$, then
$|A|^2\equiv n$. Later, Zhang \cite{Zhang1} extended the second
pinching theorem due to Peng-Terng \cite{PT2} and Wei-Xu \cite{WX}
to the case of $n=8$. In 2011, Ding and Xin \cite{DX1} verified the
following important rigidity theorem, as stated
\begin{theorem}\label{thm2}
Let $M$ be an $n$-dimensional compact minimal hypersurface in the
unit sphere $\mathbb{S}^{n+1}$. If $n\geq6$, the squared norm of the
second fundamental form satisfies $0\leq |A|^2-n\leq
\frac{n}{23}$, then $|A|^2\equiv n$, i.e., $M$ is a Clifford torus.
\end{theorem}

In the present paper, we first give a refined version of Ding-Xin's rigidity theorem for minimal hypersurfaces in a sphere.
\begin{thm}\label{mthm1}
Let $M$ be an $n$-dimensional compact minimal hypersurface in the
unit sphere $\mathbb{S}^{n+1}$. If the squared norm of the
second fundamental form satisfies $0\leq |A|^2-n\leq\frac{n}{22}$, then $|A|^2\equiv n$ and $M$ is a Clifford torus.
\end{thm}

Moreover, we consider the isometric immersion $X: M \rightarrow \mathbb{R}^{n+1}$.
If the position vector $X$ evolves in the direction of the mean curvature vector $\vec{H}$, then it gives rise
to a solution to the mean curvature flow
$$\left\{\begin{array}{llll} \frac{\partial}{\partial t}X(x, t)=\vec{H}(x, t),\,\,
x\in M,\\
X(x, 0)=X(x).
\end{array} \right.$$
An important class of solutions to the above mean curvature flow equations are
self-shrinkers, which satisfy
\begin{equation}\label{selfshr}
H=-X^N,
\end{equation}
where $X^N$ is the projection of $X$ on the unit inner normal vector $\xi$, i.e., $X^N=<X,\xi>$.

Rigidity problems of self-shrinkers have been studied extensively
\cite{CaoLi,ChengWei,CM,DX3,DX2,Hu2,LeSe,LiWei}. In 2011, Le and
Sesum \cite{LeSe} proved that any $n$-dimensional complete
self-shrinker with polynomial volume growth in $\mathbb{R}^{n+1}$
whose squared norm of the second fundamental form satisfies
$|A|^2<1$ must be a hyperplane. Afterwards, Cao and Li \cite{CaoLi}
generalized this rigidity result to arbitrary codimensional cases
and proved that if $M$ is an $n$-dimensional complete self-shrinker
with polynomial volume growth in $\mathbb{R}^{n+q}$, and if
$|A|^2\leq1$, then $M$ must be one of the generalized cylinders.
Under the assumption that $|A|$ is constant, Cheng and Wei
\cite{ChengWei} proved that if $M$ is an $n$-dimensional complete
self-shrinker with polynomial volume growth in $\mathbb{R}^{n+1}$,
and if $0\leq|A|^2-1\leq\frac{3}{7}$, then $|A|^2=1$.

In 2014, Ding and Xin \cite{DX2} proved the second pinching theorem
for self-shrinkers in the Euclidean space.
\begin{theorem}\label{thm3}
Let $M$ be an $n$-dimensional complete self-shrinker with polynomial volume growth in $\mathbb{R}^{n+1}$.
If the squared norm of the second fundamental form satisfies
$0\leq|A|^2-1\leq0.022$, then $|A|^2\equiv1$ and $M$ is a round sphere or a cylinder.
\end{theorem}

In Section \ref{sec3}, we improve Ding-Xin's pinching constant in Theorem \ref{thm3} and prove the following rigidity theorem for self-shrinkers in the Euclidean space.
\begin{thm}\label{mthm2}
Let $M$ be an $n$-dimensional complete self-shrinker with polynomial volume growth in $\mathbb{R}^{n+1}$.
If the squared norm of the second fundamental form satisfies
$0\leq|A|^2-1\leq\frac{1}{21}$, then $|A|^2\equiv1$ and $M$ is a round sphere or a cylinder.
\end{thm}

\section{Minimal hypersurfaces in the unit sphere}\label{sec2}
Let $M$ be an $n(\geq2)$-dimensional compact minimal hypersurface in the unit sphere $\mathbb{S}^{n+1}$. We shall make use of the following convention on the range of indices:
                     $$1\leq i, j, k, \ldots\leq n.$$
We choose a local orthonormal frame $\{e_1, e_2, \ldots, e_{n+1}\}$
near a fixed point $x\in M$ over $\mathbb{S}^{n+1}$ such that
$\{e_1, e_2, \ldots, e_n\}$ are tangent to $M$.
\par Let $\{\omega_1, \omega_2,\ldots, \omega_{n+1}\}$ be the dual frame fields of $\{e_1, e_2,
\ldots, e_{n+1}\}$. Denote by $R_{ijkl}$ and
$A=\sum\limits_{i,j}h_{ij}\omega_{i}\otimes\omega_{j}$ the
Riemannian curvature tensor and the second fundamental form of $M$,
respectively. Then we have $\tr A=\sum\limits_{i}h_{ii}=0$. Put
$S=|A|^2=\sum\limits_{i,j}h_{ij}^{2}$. We denote the first, the
second and the third covariant derivatives of the second fundamental
form of $M$ by
$$\nabla A=\sum\limits_{i,j,k}h_{ijk}\omega_{i}\otimes\omega_{j}\otimes\omega_{k},$$
$$\nabla^2 A=\sum\limits_{i,j,k,l}h_{ijkl}\omega_{i}\otimes\omega_{j}\otimes\omega_{k}\otimes\omega_{l},$$
$$\nabla^3 A=\sum\limits_{i,j,k,l,m}h_{ijklm}\omega_{i}\otimes\omega_{j}\otimes\omega_{k}\otimes\omega_{l}\otimes\omega_{m}.$$
We have the Gauss and Codazzi equations.
\begin{equation}\label{1.2}
R_{ijkl}=\delta_{ik}\delta_{jl}-\delta_{il}\delta_{jk}+h_{ik}h_{jl}-h_{il}h_{jk},
\end{equation}
and
\begin{equation}\label{1.3}
h_{ijk}=h_{ikj}.
\end{equation}
By the Gauss equations, the scalar curvature of $M$ is given by
$$R=n(n-1)-S.$$
We also have the Ricci identities.
\begin{equation}\label{1.4}
h_{ijkl}=h_{ijlk}+\sum\limits_{m}h_{mj}R_{mikl}+\sum\limits_{m}h_{im}R_{mjkl},
\end{equation}
\begin{equation}\label{1.5}
h_{ijklm}=h_{ijkml}+\sum\limits_{r}h_{rjk}R_{rilm}+\sum\limits_{r}h_{irk}R_{rjlm}+\sum\limits_{r}h_{ijr}R_{rklm}.
\end{equation}

Choose a suitable orthonormal frame $\{e_1, e_2, \ldots, e_n\}$  at
$x$ such that $h_{ij}=\lambda_{i}\delta_{ij}$ for all $i$, $j$. Then
we have
\begin{equation}\label{1.6}
\frac{1}{2}\Delta S=S(n-S)+|\nabla A|^2.
\end{equation}
Hence
\begin{equation}\label{1.6-1}
|\nabla S|^2=\frac{1}{2}\Delta S^2-S\Delta S=\frac{1}{2}\Delta S^2+2S^2(S-n)-2S|\nabla A|^2.
\end{equation}
We get
\begin{eqnarray}\label{1.7}
\frac{1}{2}\Delta|\nabla A|^2&=&\nonumber (2n+3-S)|\nabla
A|^2-\frac{3}{2}|\nabla S|^2+|\nabla^2
A|^2\\
& &+3(2B_2-B_1),
\end{eqnarray}
where $B_1=\sum\limits_{i,j,k,l,m}h_{ijk}h_{ijl}h_{km}h_{ml}$ and $B_2=\sum\limits_{i,j,k,l,m}h_{ijk}h_{klm}h_{im}h_{jl}$.

Following \cite{PT2} (also see \cite{DX1,XX2}), we have
\begin{equation}\label{1.8}
|\nabla^2 A|^2-\frac{3S(S-n)^2}{2(n+4)}\geq\frac{3}{4}\sum\limits_{i\neq
j}t_{ij}^2=\frac{3}{4}\sum\limits_{i,j}t_{ij}^2,
\end{equation}
\begin{equation}\label{1.9}
3(B_1-2B_2)\leq \sigma S|\nabla A|^2,
\end{equation}
where $t_{ij}=h_{ijij}-h_{jiji}$ and $\sigma=\frac{\sqrt{17}+1}{2}$.

Moreover, we have
\begin{equation}\label{1.10}
\int_{M}(B_1-2B_2)dM=\int_{M}(Sf_4-f_{3}^2-S^2-\frac{|\nabla S|^2}{4})dM,
\end{equation}
where $f_k=\tr A^k=\sum\limits_{i}\lambda_{i}^k$.

Set $G=\sum\limits_{i,j}t_{ij}^2=\sum\limits_{i,j}(\lambda_i-\lambda_j)^2(1+\lambda_i\lambda_j)^2$.
Thus, we have
\begin{equation}\label{1.11}
G=2Sf_4-2f_3^2-2S^2-2S(S-n).
\end{equation}
This implies that
\begin{equation}\label{1.12}
\frac{1}{2}\int_MGdM=\int_M[(B_1-2B_2)-|\nabla A|^2+\frac{1}{4}|\nabla S|^2]dM.
\end{equation}

Now we are in a position to prove our rigidity theorem for minimal hypersurfaces in a sphere.
\begin{proof}[Proof of Theorem \ref{mthm1}]
It follows from (\ref{1.7}), (\ref{1.8}) and (\ref{1.12}) that
\begin{equation}\label{1.13}
\int_{M}[(S-2n-\frac{3}{2})|\nabla A|^2+\frac{3}{2}(B_1-2B_2)+\frac{9}{8}|\nabla S|^2]dM\geq0.
\end{equation}
Under the pinching condition $n\leq S\leq n+\frac{n}{k}$, (\ref{1.13}) together with (\ref{1.6-1}) implies
\begin{eqnarray}\label{1.14}
0&\leq&\nonumber\int_{M}[(S-2n-\frac{3}{2})|\nabla A|^2+\frac{3}{2}(B_1-2B_2)\\
& &\nonumber+\frac{9}{4}(S^3-nS^2-S|\nabla A|^2)]dM\\
&\leq&\int_{M}\Big[\Big(\frac{k+9}{4k}n-\frac{5}{4}S-\frac{3}{2}\Big)|\nabla A|^2+\frac{3}{2}(B_1-2B_2)\Big]dM.
\end{eqnarray}

When $n=2,3$, combining (\ref{1.9}) and (\ref{1.14}), we get
\begin{equation}\label{1.15}
0\leq\int_{M}\Big(\frac{k+9}{4k}n+\frac{\sqrt{17}-4}{4}S-\frac{3}{2}\Big)|\nabla A|^2dM.
\end{equation}

When $n=4,5$, by Lemma 3.4 in \cite{Zhang1}, we have
\begin{equation}\label{1.16}
3(B_1-2B_2)\leq\eta_nS|\nabla A|^2,
\end{equation}
where $\eta_4=2.16$ and $\eta_5=2.23$. It follows from (\ref{1.14})
and (\ref{1.16}) that
\begin{equation}\label{1.17}
0\leq\int_{M}\Big(\frac{1+\frac{9}{k}}{4}n-\frac{5-2\eta_n}{4}S-\frac{3}{2}\Big)|\nabla A|^2dM.
\end{equation}

Taking $k=22$, we see that the coefficients of the integrals in
(\ref{1.15}) and (\ref{1.17}) are both negative. Therefore, $\nabla
A=0$ and $|A|^2=n$, for $n\leq5$.

When $n\geq6$, the following lemma can be found in \cite{DX1}.
\begin{lemma}\label{A-2B}
If $n\geq6$, $n\leq S\leq\frac{16}{15}n$, then
$3(B_1-2B_2)\leq(S+4+C_1(n)G^{1/3})|\nabla A|^2$,
where $C_1(n)=\Big(\frac{3-\sqrt{6}-4p}{\sqrt{6}-1+13p}(6-\sqrt{6}-13p)^2\Big)^{1/3}$ and
$p=\frac{1}{13(n-2)}$.
\end{lemma}
For $\theta\in(0, 1)$, from (\ref{1.7}), (\ref{1.8}) and (\ref{1.12}), using Lemma \ref{A-2B} and Young's inequality, we drive the following inequality.
\begin{eqnarray}\label{1.18}
&
&\nonumber\frac{3(1-\theta)}{4}\int_MGdM+\int_M\frac{3S(S-n)^2}{2(n+4)}dM\\
&\leq&\nonumber\int_M\Big[(S-2n-3)|\nabla A|^2+\frac{3}{2}|\nabla S|^2+3(B_1-2B_2)-\frac{3\theta}{4}G\Big]dM\\
&=&\nonumber\int_M(S-2n-3+\frac{3\theta}{2})|\nabla A|^2dM+(\frac{3}{2}-\frac{3\theta}{8})\int_M|\nabla S|^2dM\\
& &\nonumber+(3-\frac{3\theta}{2})\int_M(B_1-2B_2)dM\\
&\leq&\nonumber\int_M(S-2n-3+\frac{3\theta}{2})|\nabla A|^2dM+(\frac{3}{2}-\frac{3\theta}{8})\int_M|\nabla
S|^2dM\\
& &\nonumber+(1-\frac{\theta}{2})\int_M\Big(S+4+C_1(n)G^{1/3}\Big)|\nabla A|^2dM\\
&\leq&\nonumber\int_M\Big[(2-\frac{\theta}{2})S-2n+1-\frac{\theta}{2}\Big]|\nabla A|^2dM+(\frac{3}{2}-\frac{3\theta}{8})\int_M|\nabla S|^2dM\\
& &+\frac{3(1-\theta)}{4}\int_MGdM+C_3\int_M|\nabla A|^3dM,
\end{eqnarray}
where
$C_3=C_3(n,\theta)=\frac{4}{9}C_1(n)^{3/2}(1-\frac{\theta}{2})^{3/2}(1-\theta)^{-1/2}$.

For $\epsilon>0$, (\ref{1.6}) together with the divergence theorem and Cauchy-Schwarz's
inequality implies
\begin{eqnarray}\label{1.19}
\int_M|\nabla A|^3dM&\leq&\nonumber\int_MS(S-n)|\nabla A|dM+\epsilon\int_M|\nabla^2 A|^2dM\\
& &+\frac{1}{16\epsilon}\int_M|\nabla S|^2dM.
\end{eqnarray}

Under the pinching condition $n\leq S\leq n+\delta(n)(\leq\frac{n}{15})$, for $\epsilon_1=\theta_1\epsilon>0$, we estimate $\int_MS(S-n)|\nabla A|dM$ as follows.
\begin{eqnarray}\label{1.20}
\int_MS(S-n)|\nabla A|dM&\leq&\nonumber2(n+\delta)\epsilon_1\int_MS(S-n)dM\\
& &\nonumber+\frac{1}{8(n+\delta)\epsilon_1}\int_MS(S-n)|\nabla A|^2dM\\
&\leq&\Big[2(n+\delta)\epsilon_1+\frac{S-n}{8\epsilon_1}\Big]\int_M|\nabla A|^2dM.
\end{eqnarray}

From (\ref{1.7}) and (\ref{1.9}), we have
\begin{eqnarray}\label{1.21}
\int_M|\nabla^2 A|^2dM&\leq&\nonumber\frac{3}{2}\int_M|\nabla S|^2dM\\
& &+\int_M[(\sigma+1)S-2n-3]|\nabla A|^2dM.
\end{eqnarray}

Substituting (\ref{1.19}), (\ref{1.20}) and (\ref{1.21}) into (\ref{1.18}), we obtain
\begin{eqnarray}\label{1.22}
0&\leq&\nonumber\int_M\Big[(2-\frac{\theta}{2})S-2n+1-\frac{\theta}{2}\Big]|\nabla A|^2dM\\
& &\nonumber+(\frac{3}{2}-\frac{3\theta}{8})\int_M|\nabla S|^2dM-\int_M\frac{3S(S-n)^2}{2(n+4)}dM\\
& &\nonumber+C_3\Big\{\Big[2(n+\delta)\epsilon_1+\frac{S-n}{8\epsilon_1}\Big]\int_M|\nabla A|^2dM\\
& &\nonumber+\frac{3\epsilon}{2}\int_M|\nabla S|^2dM+\frac{1}{16\epsilon}\int_M|\nabla S|^2dM\\
& &\nonumber+\int_M\epsilon[(\sigma+1)S-2n-3]|\nabla A|^2dM\Big\}\\
&=&\nonumber\int_M\Big\{(2-\frac{\theta}{2})S-2n+1-\frac{\theta}{2}+\epsilon C_3[(\sigma+1)S-2n-3]\\
& &\nonumber+C_3\Big[2(n+\delta)\epsilon_1+\frac{S-n}{8\epsilon_1}\Big]\Big\}|\nabla A|^2dM-\int_M\frac{3S(S-n)^2}{2(n+4)}dM\\
& &\nonumber+\Big(3-\frac{3\theta}{4}+3\epsilon C_3+\frac{C_3}{8\epsilon}\Big)\int_M[S(S-n)^2-(S-n)|\nabla A|^2]dM\\
&\leq&\nonumber\int_M\Big\{1-\frac{\theta}{2}(n+1)+\epsilon C_3(n\sigma-n-3)+2\epsilon_1 C_3(n+\delta)\\
& &\nonumber-\Big[1-\frac{\theta}{4}+\frac{C_3}{8\epsilon}-\frac{C_3}{8\epsilon_1}+\epsilon C_3(2-\sigma)\Big](S-n)\Big\}|\nabla A|^2dM\\
& &\nonumber+\delta\Big(3-\frac{3\theta}{4}+3\epsilon C_3+\frac{C_3}{8\epsilon}-\frac{3}{2(n+4)}\Big)\int_M|\nabla A|^2dM\\
&=&\nonumber\int_M\Big\{1-\frac{\theta}{2}(n+1)+\epsilon C_3(n\sigma+n-3+5\delta)+\frac{C_3\delta}{8\epsilon}\\
& &\nonumber+\delta\Big(3-\frac{3\theta}{4}-\frac{3}{2(n+4)}\Big)+2\epsilon(\theta_1-1) C_3(n+\delta)\\
& &-\Big[1-\frac{\theta}{4}+\frac{C_3}{8\epsilon}-\frac{C_3}{8\epsilon_1}+\epsilon C_3(2-\sigma)\Big](S-n)\Big\}|\nabla A|^2dM.
\end{eqnarray}
Let $\epsilon=\sqrt{\frac{\delta}{8[\sigma n+n-3+5\delta+2(n+\delta)(\theta_1-1)]}}$,
$\theta=0.866$ and $\theta_1=0.83$. Then
$$C_3(n)=\frac{4}{9}\times0.567^{3/2}\times0.134^{-1/2}\times\sqrt{\frac{3-\sqrt{6}-4p}{\sqrt{6}-1+13p}}(6-\sqrt{6}-13p),$$
where $p=\frac{1}{13(n-2)}$.

Thus, (\ref{1.22}) is reduced to
\begin{eqnarray}\label{1.23}
0&\leq&\nonumber\Big[-0.433n+C_3\sqrt{\frac{\delta}{2}[\sigma
n+n-3+5\delta-0.34(n+\delta)]}\\&
&\nonumber+(2.3505-\frac{3}{2(n+4)})\delta+0.567\Big]\int_M|\nabla A|^2dM\\
& &+\Big[-0.7835-C_3\Big(2\epsilon-\sigma\epsilon-\frac{0.17}{6.64\epsilon}\Big)\Big]\int_M(S-n)|\nabla A|^2dM.
\end{eqnarray}
We put
$$\delta=\frac{n}{22},$$
\begin{eqnarray*}
g_1(x)&=&-0.433x+\sqrt{\frac{16\times0.567^3}{81\times0.134}}\times\sqrt{\frac{3-\sqrt{6}-\frac{4}{13(x-2)}}{\sqrt{6}-1+\frac{1}{x-2}}}\\
& &\times(6-\sqrt{6}-\frac{1}{x-2})\sqrt{\frac{x}{44}\Big(\frac{\sqrt{17}+3}{2}x-3+\frac{5x}{22}-\frac{0.34\times23x}{22}\Big)}\\
& &+\frac{2.3505x}{22}-\frac{3x}{44(x+4)}+0.567,
\end{eqnarray*}

\begin{eqnarray*}
g_2(x)&=&-0.7835-\sqrt{\frac{16\times0.567^3}{81\times0.134}}\times\sqrt{\frac{3-\sqrt{6}-\frac{4}{13(x-2)}}{\sqrt{6}-1+\frac{1}{x-2}}}\\
& &\times
(6-\sqrt{6}-\frac{1}{x-2})\Big[\sqrt{\frac{(1-\sqrt{17})^2x}{4\times176\Big(\frac{\sqrt{17}+3}{2}x-3+\frac{5x}{22}-\frac{0.34\times23x}{22}\Big)}}\\
& &-\frac{17}{664}\times\sqrt{\frac{176}{x}\Big(\frac{\sqrt{17}+3}{2}x-3+\frac{5x}{22}-\frac{0.34\times23x}{22}\Big)}\Big].
\end{eqnarray*}

Now we are in a position to prove that both $g_1(x)$ and $g_2(x)$ are negative for $x\geq6$.\\

\par Case I. Noting that
$$0.433x-\frac{2.3505x}{22}+\frac{3x}{44(x+4)}-0.567>0,\,\,\,\mbox{for}\,\,\, x\geq6,$$
we take
\begin{eqnarray*}
g_1(x)&=&\Big[\sqrt{\frac{16\times0.567^3}{81\times0.134}}\times\sqrt{\frac{3-\sqrt{6}-\frac{4}{13(x-2)}}{\sqrt{6}-1+\frac{1}{x-2}}}
(6-\sqrt{6}-\frac{1}{x-2})\\
& &\times\sqrt{\frac{x}{44}\Big(\frac{\sqrt{17}+3}{2}x-3+\frac{5x}{22}-\frac{0.34\times23x}{22}\Big)}\\
& &+0.433x-\frac{2.3505x}{22}+\frac{3x}{44(x+4)}-0.567\Big]^{-1}\\
& &\times\Big(\frac{1}{(x-2)(x+4)}\Big)^2\frac{Z(x)-W(x)}{(\sqrt{6}-1)(x-2)+1},
\end{eqnarray*}
where
\begin{eqnarray*}
Z(x)&=&\frac{0.567^3\times16}{0.134\times81\times44}[(6-\sqrt{6})(x-2)-1]^2\Big[(3-\sqrt{6})(x-2)-\frac{4}{13}\Big]\\
& &\times\Big[\frac{(11\sqrt{17}+38-0.34\times23)x}{22}-3\Big](x+4)^2x,
\end{eqnarray*}
and
\begin{eqnarray*}
W(x)&=&\Big[0.433x(x+4)-\frac{2.3505x(x+4)}{22}+\frac{3x}{44}-0.567(x+4)\Big]^2\\
& &\times[(\sqrt{6}-1)(x-2)+1](x-2)^2.
\end{eqnarray*}
We get
\begin{equation}
Z(x)-W(x)\leq Q_1(x),\,\,\,\mbox{for}\,\,\,x\geq0,
\end{equation}
where $Q_1(x)=-0.00868x^7+0.0575x^6+0.207x^5-3.126x^4+2.331x^3+30.434x^2-69.56x+40$.

By calculating the determinant of the Sylvester matrix associated to $Q_1$ and $Q_{1}^{'}$, we get
\begin{equation}
R(Q_1,Q_{1}^{'})\approx-32.12.
\end{equation}
It implies that the discriminant of $Q_1(x)$ is negative, which equals to
$\frac{(-1)^{21}}{-0.00868}R(Q_1,Q_{1}^{'})$. By a classical result for discriminant
of polynomials, there exists at least one pair of complex conjugate roots for $Q_1(x)$.
By a direct computation, we have
$$Q_1(-6)=501.124,\,\,Q_1(-5)=-166.787,\,\,Q_1(0)=40,$$
$$Q_1(1.5)=-1.74319,\,\,Q_1(2)=0.44096,\,\,Q_1(3)=-11.8077.$$
Thus, $Q_1(x)$ possesses five real roots in the interval $(-6,3)$.

Hence
$$Q_1(x)<0,\, \mbox{\ for\ } x\geq3.$$
Therefore, we have
$$g_1(x)<0,\, \mbox{\ for\ } x\geq6.$$

\par Case II. Put
$$U_1(x)=\sqrt{\frac{3-\sqrt{6}-\frac{4}{13(x-2)}}{\sqrt{6}-1+\frac{1}{x-2}}},\,\,\,U_2(x)=(6-\sqrt{6}-\frac{1}{x-2})$$
and
$$U_3(x)=\sqrt{\frac{22x}{(11\sqrt{17}+38-0.34\times23)x-66}}.$$
By a direct computation, we have
$$\frac{\partial}{\partial x}U_1(x)=\frac{35-9\sqrt{6}}{26U_1[(\sqrt{6}-1)(x-2)+1]^2},$$
$$\frac{\partial}{\partial x}U_2(x)=\frac{1}{(x-2)^2},\,\,\,\frac{\partial}{\partial x}U_3(x)=-\frac{3U_{3}^{3}(x)}{2x^2}.$$
Thus, we have
\begin{eqnarray*}
\frac{\partial}{\partial x}g_2(x)&=&\sqrt{\frac{16\times0.567^3}{81\times0.134}}\frac{\partial}{\partial x}\Big\{U_1(x)U_2(x)\\
& &\times\Big[\frac{17\sqrt{11}}{166U_3(x)}+\frac{(\sqrt{17}-1)\sqrt{11}}{88}U_3(x)\Big]\Big\}\\
&=&\sqrt{\frac{16\times0.567^3}{81\times0.134}}\Big\{\Big[\frac{(35-9\sqrt{6})U_2}{26U_1[(\sqrt{6}-1)(x-2)+1]^2}+\frac{U_1}{(x-2)^2}\Big]\\
& &\times\Big[\frac{17\sqrt{11}}{166U_3}+\frac{(\sqrt{17}-1)\sqrt{11}}{88}U_3\Big]\\
& &-\Big[\frac{(\sqrt{17}-1)\sqrt{11}}{88}-\frac{17\sqrt{11}}{166U_{3}^2}\Big]\frac{3U_1U_2U_{3}^{3}}{2x^2}\Big\}
\end{eqnarray*}
\begin{eqnarray*}
&=&\frac{\sqrt{\frac{16\times0.567^3}{81\times0.134}}[(11\sqrt{17}+38-0.34\times23)x-66]^{-2}}{U_1U_3(x-2)^2x[(\sqrt{6}-1)(x-2)+1]^2}\\
& &\times\Big\{\Big[\frac{(35-9\sqrt{6})U_2(x-2)^2}{26}+U_{1}^2[(\sqrt{6}-1)(x-2)+1]^2\Big]\\
& &\times\Big[\frac{17\sqrt{11}}{166}+\frac{(\sqrt{17}-1)\sqrt{11}}{88}U_{3}^2\Big][(11\sqrt{17}+30.18)x-66]^2x\\
& &-\frac{3U_{3}^{2}}{x}[(11\sqrt{17}+30.18)x-66]^2[(\sqrt{6}-1)(x-2)+1]^2\\
& &\times U_{1}^2\Big[\frac{(\sqrt{17}-1)\sqrt{11}U_{3}^2}{176}-\frac{17\sqrt{11}}{332}\Big]U_2(x-2)^2\Big\}\\
&=&\frac{\sqrt{\frac{16\times0.567^3}{81\times0.134}}[(11\sqrt{17}+38-0.34\times23)x-66]^{-2}}{U_1U_3(x-2)^2x[(\sqrt{6}-1)(x-2)+1]^2}R(x),
\end{eqnarray*}
where
\begin{eqnarray*}
R(x)&=&\Big\{\frac{(35-9\sqrt{6})[(6-\sqrt{6})(x-2)-1](x-2)}{26}\\
& &+\Big[(3-\sqrt{6})(x-2)-\frac{4}{13}\Big][(\sqrt{6}-1)(x-2)+1]\Big\}\\
& &\times\Big\{\frac{17\sqrt{11}}{166}[(11\sqrt{17}+30.18)x-66]^2x\\
& &+\frac{(\sqrt{17}-1)\sqrt{11}}{4}[(11\sqrt{17}+30.18)x-66]x^2\Big\}\\
& &-66\Big[(3-\sqrt{6})(x-2)-\frac{4}{13}\Big][(\sqrt{6}-1)(x-2)+1]\\
& &\times\Big\{\frac{(\sqrt{17}-1)\sqrt{11}x}{8}-\frac{17\sqrt{11}}{332}[(11\sqrt{17}+30.18)x-66]\Big\}\\
& &\times[(6-\sqrt{6})(x-2)-1](x-2).
\end{eqnarray*}
We have
\begin{equation}
R(x)\geq 10000Q_2(x),\,\,\,\mbox{for}\,\,\,x\geq0,
\end{equation}
where $Q_2(x)=0.7633x^5-5.1552x^4+13.4534x^3-17.435x^2+11.598x-3.2064$.

By calculating the determinant of the Sylvester matrix associated to $Q_2$ and $Q_{2}^{'}$, we get
\begin{equation}
R(Q_2,Q_{2}^{'})\approx-0.145.
\end{equation}
Hence, the discriminant of $Q_2(x)$ is negative.
So, there exists at least one pair of complex conjugate roots for $Q_2(x)$.

On the other hand, we get
$$Q_2(0)=-3.2064,\,\,Q_2(1)=0.0181,$$
$$Q_2(2)=-0.1808,\,\,Q_2(3)=5.8251.$$
Hence
$$Q_2(x)>0,\, \mbox{\ for\ } x\geq3.$$
It implies that $g_2(x)$ is increasing for $x\geq3$.
Notice that
$$\lim\limits_{x\rightarrow\infty}g_2(x)\approx-0.044.$$
Therefore, we have
$$g_2(x)<0,\, \mbox{\ for\ } x\geq6.$$

From(\ref{1.23}), we have $\nabla A=0$ and $S=n$, i.e., $M$ is a Clifford torus.

\end{proof}

\begin{rmk}
In fact, we can enlarge the second pinching interval in Theorem
\ref{mthm1} to $[n,\,n+\frac{n}{21.6}]$ by modifying the parameters
$\theta$ and $\theta_1$.
\end{rmk}

\section{Self-shrinkers in the Euclidean space}\label{sec3}

\par Let $M$ be an $n$-dimensional complete self-shrinker with polynomial volume growth in $\mathbb{R}^{n+1}$. We shall make use of the following convention on the range of indices:
                     $$1\leq i, j, k, \ldots\leq n.$$
We choose a local orthonormal frame $\{e_1, e_2, \ldots, e_{n+1}\}$
near a fixed point $x\in M$ over $\mathbb{R}^{n+1}$ such that
$\{e_i\}$ are tangent to $M$ and $e_{n+1}$ equals to the unit inner normal vector $\xi$.
Let $\{\omega_1, \omega_2, \ldots, \omega_{n+1}\}$ be the dual frame fields of $\{e_1, e_2, \ldots, e_{n+1}\}$.
Denote by $A$, $H$, $S$ and $R_{ijkl}$ the second fundamental form, the mean curvature, the
squared length of the second fundamental form and the Riemannian curvature tensor of $M$, respectively.
Then we have
\begin{equation}\label{2.2}
A=\sum\limits_{i,j}h_{ij}\omega_{i}\otimes\omega_{j},\,\,
h_{ij}=h_{ji},
\end{equation}
\begin{equation}\label{2.3}
S=|A|^2=\sum\limits_{i,j}h^{2}_{ij},
\end{equation}
\begin{equation}\label{2.3.1}
H=\tr A=\sum\limits_{i}h_{ii}.
\end{equation}
The Gauss and Codazzi equations are given by
\begin{equation}\label{2.4}
R_{ijkl}=h_{ik}h_{jl}-h_{il}h_{jk},
\end{equation}
and
\begin{equation}\label{2.4.2}
h_{ijk}=h_{ikj}.
\end{equation}
We have the Ricci identities.
\begin{equation}\label{2.4.3}
h_{ijkl}-h_{ijlk}=\sum\limits_{m}h_{im}R_{mjkl}+\sum\limits_{m}h_{mj}R_{mikl},
\end{equation}
\begin{equation}\label{2.4.4}
h_{ijklm}-h_{ijkml}=\sum\limits_{r}h_{rjk}R_{rilm}+\sum\limits_{r}h_{irk}R_{rjlm}+\sum\limits_{r}h_{ijr}R_{rklm}.
\end{equation}
To simplify computation, we choose the local orthonormal frame such that its tangential covariant derivatives vanish at $x\in M$, i.e., $\overline{\nabla}_{e_i}e_j=h_{ij}\xi$.

Then we have
\begin{equation}\label{2.4.5}
\nabla_{e_i}H=-\nabla_{e_i}<X,\xi>=h_{ik}<X, e_k>,
\end{equation}
and
\begin{eqnarray}\label{2.4.6}
\Hess H(e_i, e_j)&=&\nonumber-\nabla_{e_i}\nabla_{e_j}<X, \xi>\\
&=&h_{ijk}<X, e_k>+h_{ij}-Hh_{ik}h_{kj}.
\end{eqnarray}

In \cite{CM}, Colding and Minicozzi introduced the linear operator
$$\mathcal{L}=\Delta-<X, \nabla(\cdot)>=e^{\frac{|X|^2}{2}}\dv(e^{-\frac{|X|^2}{2}}\nabla(\cdot)).$$
They showed that $\mathcal{L}$ is self-adjoint respect to the measure $\rho dM$, where $\rho=e^{-\frac{|X|^2}{2}}$.

Thus, we get
\begin{eqnarray}\label{2.4.7}
\mathcal{L}|A|^2&=&\nonumber \Delta|A|^2-<X, \nabla|A|^2>\\
&=&\nonumber 2\sum\limits_{i,j}h_{ij}\Delta h_{ij}+2|\nabla A|^2-2\sum\limits_{i,j,k}h_{ij}h_{ijk}<X, e_k>\\
&=&\nonumber 2(\sum\limits_{i,j}h_{ij}\nabla_{e_i}\nabla_{e_j}H+H\sum\limits_{i,j,k}h_{ij}h_{jk}h_{ki}-|A|^4)\\
& &\nonumber+2|\nabla A|^2-2\sum\limits_{i,j,k}h_{ij}h_{ijk}<X, e_k>\\
&=&2|A|^2-2|A|^4+2|\nabla A|^2.
\end{eqnarray}
We have the following equalities for $|\nabla S|^2$ and $|\nabla^2
A|^2$.
\begin{equation}\label{2.4.7-2}
|\nabla S|^2=\frac{1}{2}\mathcal{L}S^2+2S^2(S-1)-2S|\nabla A|^2,
\end{equation}
\begin{equation}\label{2.4.8}
|\nabla^2 A|^2=\frac{1}{2}\mathcal{L}|\nabla A|^2+(|A|^2-2)|\nabla A|^2+3(B_1-2B_2)+\frac{3}{2}|\nabla S|^2,
\end{equation}
where $B_1=\sum\limits_{i,j,k,l,m}h_{ijk}h_{ijl}h_{km}h_{ml}$ and $B_2=\sum\limits_{i,j,k,l,m}h_{ijk}h_{klm}h_{im}h_{jl}$.

Following \cite{DX2,PT2}, we have
\begin{equation}\label{2.4.8-1}
3(B_1-2B_2)\leq \sigma S|\nabla A|^2,
\end{equation}
where $\sigma=\frac{\sqrt{17}+1}{2}$.

We choose a local orthonormal frame $\{e_i\}$ such that
$h_{ij}=\mu_i\delta_{ij}$ at $x$. By Ricci identity (\ref{2.4.3}),
we have
\begin{equation}\label{2.4.9}
t_{ij}:=h_{ijij}-h_{jiji}=\mu_i\mu_j(\mu_i-\mu_j).
\end{equation}

By a computation, we obtain
\begin{equation}\label{2.4.10}
\int_M(B_1-2B_2)\rho dM=\int_M\Big(\frac{1}{2}G-\frac{1}{4}|\nabla S|^2\Big)\rho dM,
\end{equation}
where $G=\sum\limits_{i,j}t^{2}_{ij}=2(Sf_4-f^{2}_{3})$ and $f_k=\tr A^k=\sum\limits_{i}\mu^{k}_{i}$.

By Lemma 4.2 in \cite{DX2}, we have
\begin{equation}\label{2.4.11}
3(B_1-2B_2)\leq(S+C_2G^{1/3})|\nabla A|^2,
\end{equation}
where $C_2=\frac{2\sqrt{6}+3}{\sqrt[3]{21\sqrt{6}+103/2}}$.

Set $u_{ijkl}=\frac{1}{4}(h_{ijkl}+h_{lijk}+h_{klij}+h_{jkli})$. Notice
that $u_{ijkl}$ is symmetric in $i,j,k,l$. Then we have
\begin{eqnarray}\label{2.4.12-1}
\sum\limits_{i,j,k,l}(h_{ijkl}^{2}-u_{ijkl}^{2})&=&
\nonumber\frac{1}{16}\sum\limits_{i,j,k,l}[(h_{ijkl}-h_{lijk})^2+(h_{ijkl}-h_{klij})^2\\
& &\nonumber +(h_{ijkl}-h_{jkli})^2+(h_{lijk}-h_{klij})^2\\
& &\nonumber +(h_{lijk}-h_{jkli})^2+(h_{klij}-h_{jkli})^2]\\
&\geq&\nonumber\frac{6}{16}\sum\limits_{i\neq j}[(h_{ijij}-h_{jiji})^2+(h_{jiji}-h_{ijij})^2]\\
&=&\frac{3}{4}G,
\end{eqnarray}
i.e.,
\begin{equation}\label{2.4.12}
|\nabla^2 A|^2\geq\frac{3}{4}G.
\end{equation}

Now we are in a position to prove our rigidity theorem for self-shrinkers in the Euclidean space.
\begin{proof}[Proof of Theorem \ref{mthm2}]
Combining (\ref{2.4.8}), (\ref{2.4.10}), (\ref{2.4.11}) and (\ref{2.4.12}), using Young's inequality,
we drive the following inequality for arbitrary $\theta\in(0, 1)$.
\begin{eqnarray}\label{2.4.13}
& &\nonumber\frac{3}{4}(1-\theta)\int_M G\rho dM+\frac{3\theta}{8}\int_M|\nabla S|^2\rho dM\\
&\leq&\nonumber\int_M |\nabla^2 A|^2\rho dM-\frac{3\theta}{2}\int_M B\rho dM\\
&=&\nonumber\int_M [(S-2)|\nabla A|^2+\frac{3}{2}|\nabla S|^2]\rho dM\\
& &\nonumber+3\Big(1-\frac{\theta}{2}\Big)\int_M(B_1-2B_2)\rho dM\\
&\leq&\nonumber\int_M [(S-2)|\nabla A|^2+\frac{3}{2}|\nabla S|^2]\rho dM\\
& &\nonumber+\Big(1-\frac{\theta}{2}\Big)\int_M(S+C_2G^{1/3})|\nabla A|^2\rho dM\\
&\leq&\nonumber\int_M \Big\{\Big[\Big(2-\frac{\theta}{2}\Big)S-2\Big]|\nabla A|^2+\frac{3}{2}|\nabla S|^2\Big\}\rho dM\\
& &+C_4\int_M|\nabla A|^3\rho dM+\frac{3}{4}(1-\theta)\int_M G\rho dM,
\end{eqnarray}
where $C_4(\theta)=\frac{4}{9}C^{3/2}_2\Big(1-\frac{\theta}{2}\Big)^{3/2}(1-\theta)^{-1/2}$.
This implies that
\begin{eqnarray}\label{2.4.14}
0&\leq&\nonumber\int_M\Big[\Big(2-\frac{\theta}{2}\Big)S-2\Big]|\nabla A|^2\rho dM\\
& &+\Big(\frac{3}{2}-\frac{3\theta}{8}\Big)\int_M|\nabla S|^2\rho dM+C_4\int_M|\nabla A|^3\rho dM.
\end{eqnarray}
Applying the divergence theorem, we get
\begin{eqnarray}\label{2.4.15}
& &\nonumber-\int_M \nabla|\nabla A|\cdot\nabla S\rho dM\\
&=&\nonumber\int_M|\nabla A|\Delta S\rho dM+\int_M|\nabla A|\nabla S\cdot\nabla\rho dM\\
&=&\nonumber\int_M|\nabla A|\Delta S\rho dM+\int_M|\nabla A|(-<X,\nabla S>)\rho dM\\
&=&\int_M|\nabla A|\mathcal{L} S\rho dM.
\end{eqnarray}
From(\ref{2.4.7}), (\ref{2.4.8}), (\ref{2.4.8-1}) and (\ref{2.4.15}), for arbitrary $\epsilon>0$, we have
\begin{eqnarray}\label{2.4.16}
\int_M|\nabla A|^3\rho dM&=&\nonumber\int_M(S^2-S+\frac{1}{2}\mathcal{L}|A|^2)|\nabla A|\rho dM\\
&=&\nonumber\int_M(S^2-S)|\nabla A|\rho dM-\frac{1}{2}\int_M \nabla|\nabla A|\cdot\nabla S\rho dM\\
&\leq&\nonumber\int_M(S^2-S)|\nabla A|\rho dM+\epsilon\int_M|\nabla^2 A|^2\rho dM\\
& &\nonumber+\frac{1}{16\epsilon}\int_M |\nabla S|^2\rho dM\\
&\leq&\nonumber\int_M(S^2-S)|\nabla A|\rho dM+(\frac{3\epsilon}{2}+\frac{1}{16\epsilon})\int_M |\nabla S|^2\rho dM\\
& &+\epsilon\int_M(\sigma S+S-2)|\nabla A|^2\rho dM.
\end{eqnarray}
From (\ref{2.4.7-2}), we have
\begin{equation}\label{2.4.17}
\frac{1}{2}\int_M |\nabla S|^2\rho dM=\int_MS(S-1)^2\rho dM-\int_M(S-1)|\nabla A|^2\rho dM.
\end{equation}
When $1\leq S\leq1+\delta$, for $\epsilon_1=\theta_1\epsilon>0$, we obtain
\begin{eqnarray}\label{2.4.18}
\int_MS(S-1)|\nabla A|\rho dM&\leq&\nonumber(2+2\delta)\epsilon_1\int_MS(S-1)\rho dM\\
& &\nonumber+\frac{1}{(8+8\delta)\epsilon_1}\int_MS(S-1)|\nabla A|^2\rho dM\\
&\leq&\nonumber(2+2\delta)\epsilon_1\int_MS(S-1)\rho dM\\
& &+\frac{1}{8\epsilon_1}\int_M(S-1)|\nabla A|^2\rho dM.
\end{eqnarray}
Substituting (\ref{2.4.16}), (\ref{2.4.17}) and (\ref{2.4.18}) into
(\ref{2.4.14}), we have
\begin{eqnarray}
0&\leq&\nonumber\int_M\Big[\Big(2-\frac{\theta}{2}\Big)S-2\Big]|\nabla A|^2\rho dM+\Big(\frac{3}{2}-\frac{3\theta}{8}\Big)\int_M|\nabla S|^2\rho dM\\
& &\nonumber+C_4\Big[(2+2\delta)\epsilon_1\int_MS(S-1)\rho dM+\frac{1}{8\epsilon_1}\int_M(S-1)|\nabla A|^2\rho dM\\
& &\nonumber+\epsilon\int_M(\sigma S+S-2)|\nabla A|^2\rho dM+\Big(\frac{3\epsilon}{2}+\frac{1}{16\epsilon}\Big)\int_M |\nabla S|^2\rho dM\Big]
\end{eqnarray}
\begin{eqnarray}\label{2.4.20}
&\leq&\nonumber\int_M\Big[\Big(2-\frac{\theta}{2}\Big)S-2+\frac{C_4}{8\epsilon_1}(S-1)+2C_4\epsilon_1(1+\delta)\\
& &\nonumber+C_4\epsilon(\sigma S+S-2)\Big]|\nabla A|^2\rho dM+\Big[3-\frac{3\theta}{4}+C_4(3\epsilon+\frac{1}{8\epsilon})\Big]\\
& &\nonumber\times\Big[\int_M\delta S(S-1)\rho dM-\int_M(S-1)|\nabla A|^2\rho dM\Big]\\
&=&\nonumber\int_M\Big[-\frac{\theta}{2}+\frac{\sqrt{17}+3}{2}C_4\epsilon+\Big(5\epsilon+\frac{1}{8\epsilon}\Big)\delta C_4\\
& &\nonumber+\Big(3-\frac{3\theta}{4}\Big)\delta+2C_4(1+\delta)(\epsilon_1-\epsilon)\Big]|\nabla A|^2\rho dM\\
& &-\int_M\Big[1-\frac{\theta}{4}-\frac{\sqrt{17}-3}{2}C_4\epsilon+\frac{C_4}{8\epsilon}-\frac{C_4}{8\epsilon_1}\Big](S-1)|\nabla A|^2\rho dM.
\end{eqnarray}

Let $\epsilon_1=\theta_1\epsilon$, $\theta_1=0.81$,
$\epsilon=\sqrt{\frac{\delta}{4(\sqrt{17}+3)+16(1+\delta)(\theta_1-1)+40\delta}}$, $\theta=0.836$. Then we have $C_4\leq1.066218$. Thus, (\ref{2.4.20}) is reduced to
\begin{eqnarray}\label{2.4.21}
0&\leq&\nonumber\Big[\frac{C_4}{4}\sqrt{40\delta^2-3.04(1+\delta)\delta+4(\sqrt{17}+3)\delta}\\
& &\nonumber-0.418+2.373\delta\Big]\int_M|\nabla A|^2\rho dM\\
& &\nonumber+\Big[\frac{C_4}{8}\Big(4\sqrt{17}\epsilon-12\epsilon-\frac{1}{\epsilon}+\frac{1}{0.81\epsilon}\Big)-0.791\Big]\\
& &\times\int_M(|A|^2-1)|\nabla A|^2\rho dM.
\end{eqnarray}
We take $\delta=1/21$. Then the coefficients of the integrals in
(\ref{2.4.21}) are both negative. Therefore, we have $|\nabla
A|\equiv0$ and $|A|\equiv1$, i.e., $M$ is a round sphere or a
cylinder.
\end{proof}

To attack the Chern conjecture, we would like to suggest a program
that divides the second optimal pinching problem for minimal
hypersurfaces in a sphere into several steps. Precisely, we propose
the following unsolved problem.
\begin{prob}
Let $M$ be a compact minimal hypersurface in $\mathbb{S}^{n+1}$.
Denote by $A$ the second fundamental form of $M$. Is it possible to
prove that there exists a positive integer $k$ with $1\le k\le 21$,
such that if $0\leq|A|^2-n\le\frac{n}{k}$, then $|A|^2\equiv n$ and
$M$ is a Clifford torus, or $|A|^2\equiv 2n$ and $k=1$?
\end{prob}
For self-shrinkers in the Euclidean space, we would like to propose
the following problem.
\begin{prob}
Let $M$ be an $n$-dimensional complete self-shrinker with polynomial volume growth in $\mathbb{R}^{n+1}$.
Denote by $A$ the second fundamental form of $M$. Is it possible to prove that there exists a
positive integer $k$ with $1\le k\le 20$, such that if
$0\leq|A|^2-1\le\frac{1}{k}$, then $|A|^2\equiv 1$ and $M$ is a
round sphere or a cylinder, or $|A|^2\equiv 2$ and $k=1$?
\end{prob}

\bibliographystyle{amsplain}

\begin{thebibliography}{10}

\bibitem{AB} S. C. de Almeida and F. G. B. Brito, Closed $3$-dimensional
hypersurfaces with constant mean curvature and constant scalar
curvature, \emph{Duke Math. J.}, {\bf 61}(1990), 195-206.

\bibitem{CaoLi} H. D. Cao and H. Z. Li, A gap theorem for self-shrinkers of the mean curvature flow in arbitrary
codimension, \emph{Calc. Var. Partial Differ. Equations}, {\bf
46}(2013), 879-889.

\bibitem{Ca} E. Cartan, Sur des familles remarquables d'hypersurfaces isoparam\'etriques dans les espaces sph\'eriques,
\emph{Math. Z.}, {\bf 45}(1939), 335-367.

\bibitem{C} S. P. Chang, On minimal hypersurfaces with constant scalar
curvatures in $S^4$, \emph{J. Differ. Geom.}, {\bf 37}(1993),
523-534.

\bibitem{C2} S. P. Chang, A closed hypersurface of constant scalar
  and mean curvatures in $S^4$ is isoparametric, \emph{Comm. Anal.
  Geom.}, {\bf 1}(1993), 71-100.

\bibitem{C3} S. P. Chang, On closed hypersurfaces of constant scalar curvatures and mean curvatures in $S^{n+1}$, \emph{Pacific J. Math.}, {\bf 165}(1994), 67-76.

\bibitem{ChengNa} Q. M. Cheng and H. Nakagawa, Totally umbilic hypersurfaces, \emph{Hiroshima Math. J.}, {\bf 20}(1990), 1-10.

\bibitem{ChengWei} Q. M. Cheng and G. X. Wei, A gap theorem of self-shrinkers, \emph{Trans. Amer. Math. Soc.},
{\bf 367}(2015), 4895-4915.

\bibitem{Cheng} S. Y. Cheng, On the Chern conjecture for minimal hypersurface
with constant scalar curvatures in the spheres, \emph{Tsing Hua
Lectures on Geometry and Analysis}, International Press, Cambridge,
MA, 1997, 59-78.

\bibitem{Chern} S. S. Chern, Minimal submanifolds in a Riemannian manifold. University of Kansas, Department of Mathematics Technical Report 19, Univ. of Kansas, Lawrence, Kan., 1968.

\bibitem{CDK} S. S. Chern, M. do Carmo and S. Kobayashi, Minimal submanifolds
of a sphere with second fundamental form of constant length,
\emph{Functional Analysis and Related Fields}, Springer-Verlag,
Berlin, 1970, 59-75.

\bibitem{CM} T. H. Colding and W. P. Minicozzi II, Generic mean curvature flow I: generic singularities,
\emph{Ann. of Math.}, {\bf 175}(2012), 755-833.

\bibitem{D} Qing Ding, On spectral characterizations of minimal hypersurfaces in a sphere,
\emph{Kodai Math. J.}, {\bf 17}(1994), 320-328.

\bibitem{DX1} Qi Ding and Y. L. Xin, On Chern's problem for rigidity of minimal hypersurfaces in the spheres,
\emph{Adv. Math.}, {\bf 227}(2011), 131-145.

\bibitem{DX3} Qi Ding and Y. L. Xin, Volume growth, eigenvalue and compactness for self-shrinkers, \emph{Asian J. Math.}, {\bf 17}(2013), 443-456.

\bibitem{DX2} Qi Ding and Y. L. Xin, The rigidity theorems of self-shrinkers, \emph{Trans. Amer. Math. Soc.}, {\bf 366}(2014), 5067-5085.

\bibitem{GeTang} J. Q. Ge and Z. Z. Tang, Chern conjecture and isoparametric hypersurfaces, Differential
geometry, \emph{Adv. Lect. Math.}, {\bf 22}, International Press,
Somerville, MA, 2012, 49-60.

\bibitem{GXXZ} J. R. Gu, H. W. Xu, Z. Y. Xu and E. T. Zhao, A survey on rigidity problems in geometry and topology of submanifolds, Proceedings of the 6th International Congress of Chinese Mathematicians, \emph{Adv. Lect. Math.}, {\bf
37}, Higher Education Press \& International Press, Beijing-Boston,
2016, 79-99.

\bibitem{Hu1} G. Huisken, Flow by mean curvature of convex surfaces into spheres, \emph{J. Differ. Geom.}, {\bf 22}(1984), 237-266.

\bibitem{Hu2} G. Huisken, Asymptotic behavior for singularities of the mean curvature flow, \emph{J. Differ. Geom.}, {\bf 31}(1990), 285-299.

\bibitem{L1} B. Lawson, Local rigidity theorems for minimal hypersurfaces,
\emph{Ann. of Math.}, {\bf 89}(1969), 187-197.

\bibitem{LeSe} Nam Q. Le and N. Sesum, Blow-up rate of the mean curvature during the mean curvature flow
and a gap theorem for self-shrinkers, \emph{Comm. Anal. Geom.}, {\bf
19}(2011), 1-27.

\bibitem{LL} A. M. Li and J. M. Li, An intrinsic rigidity theorem for minimal submanifolds in a sphere,
\emph{Arch. Math.}, {\bf 58}(1992), 582-594.

\bibitem{LiWei} H. Z. Li and Y. Wei, Classification and rigidity of self-shrinkers in the mean curvature flow, \emph{J. Math. Soc. Japan},
{\bf 66}(2014), 709-734.

\bibitem{Lu} Z. Q. Lu, Normal scalar curvature conjecture and its applications, \emph{J. Funct. Anal.}, {\bf 261}(2011), 1284-1308.

\bibitem{M} H. F. M\"unzner, Isoparametrisehe hyperfl\"achen in
sph\"aren, I, II, \emph{Math. Ann.}, {\bf 251, 256}(1980, 1981),
57-71, 215-232.

\bibitem{Mu} H. Muto, The first eigenvalue of the Laplacian
of an isoparametric minimal hypersurface in a unit sphere,
\emph{Math. Z.}, {\bf 197}(1988), 531-549.

\bibitem{PT1} C. K. Peng and C. L. Terng, Minimal hypersurfaces of sphere with constant scalar
curvature, \emph{Ann. of Math. Stud.}, {\bf 103}, Princeton Univ.
Press, Princeton, NJ, 1983, 177-198.

\bibitem{PT2} C. K. Peng and C. L. Terng, The scalar curvature of minimal hypersurfaces in
spheres, \emph{Math. Ann.}, {\bf 266}(1983), 105-113.

\bibitem{SWY} M. Scherfner, S. Weiss and S. T. Yau, A review of the Chern conjecture for isoparametric hypersurfaces in spheres, Advances in Geometric Analysis, \emph{Adv. Lect. Math.}, {\bf 21}, International Press, Somerville, MA, 2012, 175-187.

\bibitem{Shen} Y. B. Shen, On intrinsic rigidity for minimal submanifolds in a sphere, \emph{Sci. China Ser. A},
{\bf 32}(1989), 769-781.

\bibitem{S} J. Simons, Minimal varieties in Riemannian manifolds, \emph{Ann. of Math.}, {\bf 88}(1968),
62-105.

\bibitem{SY} Y. J. Suh. and H. Y. Yang, The scalar curvature of minimal
hypersurfaces in a unit sphere, \emph{Comm. Contemp. Math.}, {\bf
9}(2007), 183-200.

\bibitem{Ver} L. Verstraelen, Sectional curvature of minimal submanifolds, In: Proceedings
Workshop on Differential Geometry, Univ. Southampton, 1986, 48-62.

\bibitem{WX} S. M. Wei and H. W. Xu, Scalar curvature of minimal hypersurfaces
in a sphere, \emph{Math. Res. Lett.}, {\bf 14}(2007), 423-432.

\bibitem{X0} H. W. Xu, Pinching theorems, global pinching theorems and eigenvalues for Riemannian
  submanifolds, Ph.D. dissertation, Fudan University, 1990.

\bibitem{X1} H. W. Xu, A rigidity theorem for submanifolds with parallel mean curvature in a sphere, \emph{Arch. Math.}, {\bf 61}(1993),
489-496.

\bibitem{XT} H. W. Xu and L. Tian, A new pinching theorem for closed
hypersurfaces with constant mean curvature in $S^{n+1}$, \emph{Asian
J. Math.}, {\bf 15}(2011), 611-630.

\bibitem{XX1} H. W. Xu and Z. Y. Xu, The second pinching theorem for hypersurfaces with constant mean curvature in a
sphere, \emph{Math. Ann.}, {\bf 356}(2013), 869-883.

\bibitem{XX2} H. W. Xu and Z. Y. Xu, A new characterization of the Clifford torus via scalar curvature pinching,
\emph{J. Funct. Anal.}, {\bf 267}(2014), 3931-3962.

\bibitem{XX3} H. W. Xu and Z. Y. Xu, A refined version of the second pinching theorem for CMC hypersurfaces in a sphere, in preparation.

\bibitem{YC0} H. C. Yang and Q. M. Cheng, A note on the pinching constant of minimal hypersurfaces with
constant scalar curvature in the unit sphere, \emph{Kexue Tongbao},
{\bf 35}(1990), 167-170; \emph{Chinese Sci. Bull.}, {\bf 36}(1991),
1-6.

\bibitem{YC1} H. C. Yang and Q. M. Cheng, An estimate of the pinching constant of minimal hypersurfaces
with constant scalar curvature in the unit sphere, \emph{Manuscripta
Math.}, {\bf 84}(1994), 89-100.

\bibitem{YC2} H. C. Yang and Q. M. Cheng, Chern's conjecture on minimal hypersurfaces,
\emph{Math. Z.}, {\bf 227}(1998), 377-390.

\bibitem{Y1} S. T. Yau, Submanifolds with constant mean curvature,  I,  II, \emph{Amer. J. Math.}, {\bf 96, 97}(1974, 1975), 346-366,
76-100.

\bibitem{Y2} S. T. Yau, Problem section, In: Seminar on Differential Geometry, \emph{Ann. of Math. Stud.}, {\bf 102},
Princeton Univ. Press, Princeton, NJ, 1982, 669-706.

\bibitem{Zhang1} Q. Zhang, The pinching constant of minimal hypersurfaces in the unit
spheres, \emph{Proc. Amer. Math. Soc.}, {\bf 138}(2010), 1833-1841.

\end{thebibliography}

\end{document}